\documentclass[12pt,a4paper]{amsart}

\usepackage{fullpage,color}
\usepackage[utf8]{inputenc}
\usepackage{amsmath,amsfonts,amssymb,amsthm,amscd}
\usepackage{graphicx}
\usepackage{verbatim}

\theoremstyle{plain}
\newtheorem{theorem}{Theorem}[section]
\newtheorem{proposition}[theorem]{Proposition}
\newtheorem{lemma}[theorem]{Lemma}
\newtheorem{corollary}[theorem]{Corollary}

\theoremstyle{definition}

\newtheorem{example}[theorem]{Example}

\theoremstyle{remark}
\newtheorem{remark}[theorem]{Remark}

\newcommand{\C}{\mathbb C}

\newcommand{\K}{\mathbb K}
\newcommand{\N}{\mathbb N}
\newcommand{\Q}{\mathbb Q}
\newcommand{\R}{\mathbb R}
\newcommand{\Z}{\mathbb Z}

\DeclareMathOperator{\den}{den}
\DeclareMathOperator{\lcm}{lcm}
\DeclareMathOperator{\ord}{ord}

\usepackage{hyperref}

\title{Euler's factorial series at algebraic integer points}

\author{Louna Sepp\"al\"a}
\address{Louna Sepp\"al\"a, Matematiikka, PL 8000, 90014 Oulun yliopisto, Finland}
\email{louna.seppala@oulu.fi}

\date{\today}



\begin{document}

\begin{abstract}
We study a linear form in the values of Euler's series $F(t)=\sum_{n=0}^\infty n!t^n$ at algebraic integer points $\alpha_1, \ldots, \alpha_m \in \Z_\K$ belonging to a number field $\mathbb{K}$. Let $v|p$ be a non-Archimedean valuation of $\mathbb{K}$.
Two types of non-vanishing results for the linear form $\Lambda_v = \lambda_0 + \lambda_1 F_v(\alpha_1) + \ldots + \lambda_m F_v(\alpha_m)$, $\lambda_i \in \Z_{\K}$, are derived, the second of them containing a lower bound for the $v$-adic absolute value of $\Lambda_v$. The first non-vanishing result is also extended to the case of primes in residue classes.
On the way to the main results, we present explicit Pad\'e approximations to the generalised factorial series $\sum_{n=0}^\infty \left( \prod_{k=0}^{n-1} P(k) \right) t^n$, where $P(x)$ is a polynomial of degree one.
\end{abstract}

\maketitle

\section{Introduction}

Euler's factorial series
\begin{equation}\label{hyperseries}
F(t):={}_2F_0(1,1\mid t)=\sum_{n=0}^\infty n!t^n
\end{equation}
converges $p$-adically for all primes $p$ when $|t|_p \le 1$.
In the $v$-adic metric (where $v$ extends $p$ for some prime $p$) of a number field $\K$, the series $F(t)$ converges to a point in the $v$-adic closure of $\K$ (denoted by $\K_v$) when $t \in \mathbb{K}$ is such that $|t|_v < p^{\frac{1}{p-1}}$.
Thus we write $\sum_{n=0}^\infty n!t^n =: F_v(t)$ when treating the series as a function in the $v$-adic domain $\mathbb{K}$.

Euler's series \eqref{hyperseries} is a member of the class of $F$-series (series of the form $\sum_{n=0}^\infty a_n n! z^n$, with certain conditions on the coefficients $a_n$) introduced by V. G. Chirski\u\i \, in \cite{Chirskii1989}, \cite{Chirskii1990}.
In those papers he answered the problem of the existence of global relations\footnote{Let $P \in \K[x_1,\ldots,x_m]$ be a polynomial of $m$ variables and suppose $F_1(t),\ldots,F_m(t) \in \K[[t]]$ are power series. Take a $\xi \in \K$. A relation
$$
P(F_1(\xi),\ldots,F_m(\xi)) =0
$$
is called \emph{global} if it holds in all the fields $\K_v$ where all the series $F_1(\xi),\ldots,F_m(\xi)$ converge.} between members of the class of $F$-functions.
As he points out in \cite{Chirskii1992}, the results can be refined in terms of estimating the prime $p$ for which there exists a valuation $v|p$ breaking the global relation.
These estimates were made entirely effective by Bertrand, Chirski\u\i, and Yebbou in \cite{Bertrand}. In \cite[Theorem 1.1]{Bertrand} they describe an infinite collection of intervals each containing a prime number $p$ such that for some valuation $v|p$ it holds
\begin{equation}\label{BCYlinform}
h_1 f_1(\xi) + \ldots + h_m f_m(\xi) \neq 0,
\end{equation}
where $h_i \in \Z_\K$ and $f_1(t) \equiv 1, f_2(t), \ldots, f_m (t)$ are $F$-series that are linearly independent over $\K(z)$ and constitute a solution to a differential system $D$, and $\xi \in \K \setminus \{0\}$ is an ordinary point of the system $D$. What is more, the non-vanishing in \eqref{BCYlinform} is replaced by a lower bound for $\left| h_1 f_1(\xi) + \ldots + h_m f_m(\xi) \right|_v$.

In their recent paper \cite{MatZud}, T. Matala-aho and W. Zudilin studied the irrationality of $F_p(\xi)$ at a point $\xi \in \Z \setminus \{0\}$ (i.e. global relations of the numbers $1$ and $F_p(\xi)$). In Theorem \ref{mainresult} of this paper we generalise their idea to a linear form
$$
\Lambda_v := \lambda_0 + \lambda_1 F_v(\alpha_1) + \ldots + \lambda_m F_v(\alpha_m), \quad \lambda_i \in \Z_\K,
$$
in values of Euler's series at $m$ given pairwise distinct algebraic integer points $\alpha_1, \ldots, \alpha_m \in \Z_\mathbb{K} \setminus \{0\}$. 
Theorem \ref{mainresult} states that in any collection $V$ of non-Archimedean valuations of $\K$ satisfying a certain condition, there exists a valuation $v' \in V$ such that $\Lambda_{v'} \neq 0$.
The result can also be extended to the case of primes in arithmetic progressions, generalising the recent result of \cite{ATL2018}. This is done in Theorem \ref{ResidueThm} of Section \ref{sec:corollaries}.

In the second main result, Theorem \ref{mainresult2}, we characterise an interval $I(m,H)$ (where $H$ is an upper bound for the height of the coefficients $\lambda_i$) from which one can find a prime $p$ such that there exists a valuation $v'|p$ for which
$$
\| \Lambda_{v'} \|_{v'} > H^{-(m+1) - 114 m^2 \cdot  \frac{\log \log \log H}{\log \log H}}.
$$
Our method is based on explicit Pad\'e approximations, whereas Bertrand, Chirski\u\i, and Yebbou \cite{Bertrand} use Siegel's lemma. In addition, the functional dependence on $H$ in our lower bound is improved compared to \cite{Bertrand}.

The proofs of both our main results rely on Pad\'e approximations which are used to construct small approximation forms for the values $F_v(\alpha_j)$, $j=1,\ldots,m$. Therefore, before moving to the proofs of the theorems, we shall present explicit Pad\'e approximations (with the orders of the remainders as free parameters) to the generalised factorial series
\begin{equation}\label{genfacser}
G(t) = \sum_{n=0}^\infty [P]_n t^n,
\end{equation}
where $P(x)$ is a polynomial of degree one and $[P]_n := \prod_{k=0}^{n-1} P(k)$ (see Theorem \ref{Pade}).

A brief outline of the proofs is presented right after the formulation of the main results in Section \ref{sec:results}.
In addition to Theorem \ref{ResidueThm}, the last section contains some examples of the use of the main results. We shall study the sum
$$
\sum_{n=0}^\infty n! f_n,
$$
where $(f_n)_{n=0}^\infty$ is the sequence of Fibonacci numbers, and show that for any rational number $\frac{a}{b} \in \Q$ there exists a valuation of the field $\Q \! \left( \sqrt{5} \right)$ such that
$$
\sum_{n=0}^\infty n! f_n \neq \frac{a}{b}.
$$

\section{Preliminaries: Number fields and valuations}


Let $\mathbb{K}=\Q(\gamma)$ be an algebraic number field of degree $\kappa$, and let $\Z_{\mathbb{K}}$ be its ring of integers (the algebraic integers contained in $\mathbb{K}$). All the absolute values of $\mathbb{K}$ are extensions of the absolute values of $\Q$. When $p \in \mathbb{P} \cup \{\infty\}$, where $\mathbb{P}$ is the set of prime numbers, there are as many distinct extensions of $|\cdot |_p$ to $\K$ as there are irreducible factors of the minimal polynomial of $\gamma$ in $\Q_p[x]$ (see \cite[Chapter V]{Bachman}). Here $\Q_p$ denotes the topological closure of $\Q$ with respect to the metric $|\cdot|_p$, so that $\Q_\infty = \R$.

If $| \cdot |_v$ extends the standard $p$-adic metric $| \cdot |_p$ to $\mathbb{K}$, it is customary to write $v|p$, and similarly, when extending the Archimedean absolute value $| \cdot | = | \cdot |_\infty$, we write $v|\infty$. The collection of non-Archimedean valuations of $\K$ is denoted by $V_0$, and the collection of Archimedean valuations of $\K$ by $V_\infty$.

The topological closure of $\mathbb{K}$ with respect to the metric $| \cdot |_v$ is denoted by $\mathbb{K}_v$.
We also denote $\kappa_v = \left[\K_v : \Q_p\right]$ (local degree), so that $\sum_v \kappa_v = \kappa = [\K : \Q]$.

\subsection{Normalisation}

If $v|p$ for a prime $p$, there is an element $\pi \in \K$ such that $|\pi|_v < 1$ and $\langle |\pi|_v \rangle = |\K \setminus \{0\}|_v$. Then $p = u \pi^e$, where $u$ is a unit of $\K$ and $e =e_v(\K,\Q)=[|\K|_v:|\Q|_v]$ is the ramification index of the extension. It follows naturally from $|p|_p=\frac{1}{p}$ that
$$
|p|_v=\frac{1}{p}, \quad |\pi|_v=\frac{1}{p^\frac{1}{e}}
$$
However, it is convenient to use the normalisation
\begin{equation*}
\|p\|_v 
=|p|_v^{\frac{\kappa_v}{\kappa}}, \quad 
\|\pi\|_v 
=|\pi|_v^{\frac{\kappa_v}{\kappa}}.
\end{equation*}

Similarly, if $v|\infty$ and corresponds to the $i$th conjugate field $\K^{(i)}$, then we set
$$
\|x\|_v = \left|x^{(i)}\right|^{\frac{\kappa_v}{\kappa}},
$$
where $x^{(i)}$ is the $i$th conjugate of $x \in \K$. Since $\frac{\kappa_v}{\kappa} \le 1$, the triangle inequality is valid also for the normalised Archimedean absolute value.

\subsection{Product formula}

The following product formula holds for any $x \in \mathbb{K} \setminus \{0\}$:
\begin{equation}\label{prodformula}
\prod_v \|x\|_v =1,
\end{equation}
where the product is taken over all normalised, pairwise non-equivalent valuations of $\mathbb{K}$.
Note that if $x \in \Q$, then
\begin{equation}\label{valuationproperty}
\prod_{v|p} \|x\|_v = |x|_p
\end{equation}
for any $p \in \mathbb{P} \cup \{\infty\}$.

For more details on valuations, the reader is advised to consult \cite{Bachman} and \cite{Lang}.

\section{Results}\label{sec:results}

Let $m \in \Z_{\ge 1}$ and choose $m$ pairwise distinct, non-zero algebraic integers $\alpha_1, \ldots, \alpha_m \in \Z_\K \setminus \{0\}$. Denote $\overline{\alpha} = \left( \alpha_1, \ldots, \alpha_m \right)^T$. We define
$$
c_1 = c_1(\overline{\alpha}) = \prod_{v \in V_\infty} \left( \left( \max_{1 \le j \le m} \left\{ 1, \left\| \alpha_j \right\|_v \right\} \right)^m \prod_{i=1}^m \left(\|\alpha_i\|_v + \max_{1 \le j \le m} \left\{ 1, \left\| \alpha_j \right\|_v \right\} \right) \right)
$$
and
$$
c_2 = c_2 \left( \overline{\alpha}, V\right) = c_1 \prod_{v \in V} \max_{1 \le j \le m} \left\{ \left\| \alpha_j \right\|_v \right\}
$$
for any $V \subseteq V_0$

\begin{theorem}\label{mainresult}
Let $\lambda_0, \lambda_1, \ldots, \lambda_m \in \Z_{\mathbb{K}}$ be such that $\lambda_j \neq 0$ for at least one $j$.
Suppose $V \subseteq V_0$ is a collection of non-Archimedean valuations of $\mathbb{K}$ such that
\begin{equation}\label{limsupehto}
\limsup_{l \to \infty}
c_2^l (ml+m)^\kappa (ml+m)! \prod_{v \in V} \| (ml)! l! \|_v
=0.
\end{equation}
Then there exists a valuation $v' \in V$ for which
$$
\lambda_0 + \lambda_1 F_{v'}(\alpha_1) + \ldots + \lambda_m F_{v'}(\alpha_m) \neq 0.
$$
\end{theorem}

\begin{remark}
Let us show that any collection $V \subseteq V_0$ whose complement in $V_0$ is finite satisfies condition \eqref{limsupehto}. Choose $v_1, \ldots, v_k \in V_0$ and let $V=V_0 \setminus \{v_1, \ldots, v_k\}$. Suppose in addition that $v_i|p_i$ for some $p_i \in \mathbb{P}$, $i=1,\ldots,k$. Then, by recalling that
\begin{equation}\label{kertoma-arvio}
|n!|_p \ge p^{-\frac{n}{p-1}},
\end{equation}
and using the product formula \eqref{prodformula}, we get
\begin{align*}
&c_2^l (ml+m)^\kappa (ml+m)! \prod_{v \in V} \| (ml)! l! \|_v \\
= \;&\frac{c_2^l (ml+m)^\kappa (ml+m)!}{\left( \prod_{i=1}^k \|(ml)!l!\|_{v_i} \right) \prod_{v \in V_\infty} \| (ml)! l! \|_v} \\
= \;&\frac{c_2^l (ml+m)^\kappa (ml+m)!}{\left( \prod_{i=1}^k |(ml)!l!|_{p_i}^{\frac{\kappa_{v_i}}{\kappa}} \right) (ml)! l!} \\
\le \;&\frac{c_2^l \left( \prod_{i=1}^k p_i^{\frac{\kappa_{v_i}}{\kappa} \cdot \frac{ml+l}{p_i-1}} \right) (ml+m)^\kappa (ml+m)!}{(ml)! l!} \\
= \;&\frac{\left( c_2 \prod_{i=1}^k p_i^{\frac{\kappa_{v_i}}{\kappa} \cdot \frac{m+1}{p_i-1}} \right)^l (ml+m)^\kappa (ml+1)\cdots(ml+m)}{l!} \overset{l \to \infty}{\to} 0.
\end{align*}
\end{remark}

\begin{remark}
From the previous remark it follows that there are infinitely many valuations $v \in V_0$ such that $\Lambda_v \neq 0$.
\end{remark}

\begin{theorem}\label{mainresult2}
Let $\log H \ge se^s$, where $s = \max \left\{ e^\kappa +1, c_1+1, (m+3)^2+1 \right\}$, $\kappa = [\K : \Q]$. Suppose that $\lambda_0, \lambda_1, \ldots \lambda_m \in \Z_\K$ are such that at least one of them is non-zero and
\begin{equation*}
\prod_{v \in V_\infty} \max_{0 \le i \le m} \left\{ \| \lambda_i \|_v \right\} \le H.
\end{equation*}
Then there exists a prime
\begin{equation*}\label{interval}
p \in \left] \log \left( \frac{\log H}{\log \log H} \right), \frac{17m\log H}{\log \log H} \right[
\end{equation*}
and a valuation $v'|p$ for which
\begin{equation}\label{lowerbound}
\left\| \lambda_0 + \lambda_1 F_{v'}(\alpha_1) + \ldots + \lambda_m F_{v'}(\alpha_m) \right\|_{v'} > H^{-(m+1) - 114 m^2 \cdot  \frac{\log \log \log H}{\log \log H}}.
\end{equation}
\end{theorem}

The idea behind the following proofs is to use Pad\'e approximations to construct small linear forms
$$
s_{l,\mu,j} = b_{l,\mu,0} F_v(\alpha_j) - b_{l,\mu,j}, \quad b_{l,\mu,0}, b_{l,\mu,j} \in \Z_{\K}, \quad j=1,\ldots,m,
$$
in the numbers $F(\alpha_j)$. (Here $l \in \Z_{\ge 1}$ and $\mu \in \{0,1,\ldots,m\}$ are auxiliary parameters.) With these equations the linear form
$$
\Lambda_v = \lambda_0 + \lambda_1 F_v(\alpha_1) + \ldots + \lambda_m F_v(\alpha_m)
$$
under study can be written as
\begin{equation}\label{bLambda}
b_{l,\mu,0} \Lambda_v = W + \lambda_1 s_{l,\mu,1} + \ldots + \lambda_m s_{l,\mu,m},
\end{equation}
where $W=W(l,\mu) = \sum_{i=0}^m \lambda_i b_{l,\mu,i}$ is an integer element in $\K$. In case it is non-zero, the product formula implies
\begin{equation}\label{W-prod}
1 = \prod_{v} \|W\|_v.
\end{equation}

In the proof of Theorem \ref{mainresult}, we shall assume that $\Lambda_v = 0$ for all $v \in V$ (whence equation \eqref{bLambda} gives $W$ another representation as a linear combination of $s_{l,\mu,i}$), and then aim at a contradiction by estimating the product $\prod_{v} \|W\|_v$ from above. For this we need estimates for the Pad\'e coefficients $b_{l,\mu,i}$, $s_{l,\mu,i}$, expressed in terms of the auxiliary parameter $l$. These are very roughly
$$
\| b_{l,\mu,i} \|_v \approx (ml)!, \quad i=0,1,\ldots,m, \quad v|\infty,
$$
$$
\| s_{l,\mu,j} \|_v \approx (ml)!l!, \quad j=1,\ldots,m, \quad v \in V_0.
$$
The contradiction with \eqref{W-prod} is reached via the condition \eqref{limsupehto} when $l$ is taken to infinity.

When the target is a precise lower bound for $\| \Lambda_v \|_v$, the use of the parameter $l$ also becomes more subtle: We define the number $\ell$ so that it is the largest $l$ for which the expression
$$
N(l) \approx \log H + m l \log \log l - l \log l
$$
is still positive. Then we make the assumption that
$$
\left\| b_{\ell + 1,\mu,0} \Lambda_v \right\|_v < \left\| \lambda_1 s_{\ell + 1,\mu,1} + \ldots + \lambda_m s_{\ell + 1,\mu,m} \right\|_v
$$
for all $v | p$, $p \in [\log (\ell + 1),m(\ell + 2)] \cap \mathbb{P}$. This leads to the estimate
\begin{equation*}
0 \le \log \left( \prod_{v} \|W(\ell+1,\mu)\|_v \right) \approx \log H + m (\ell+1) \log \log (\ell+1) - (\ell+1) \log (\ell+1) < 0,
\end{equation*}
giving the desired contradiction. It follows that there exists a prime
\begin{equation}\label{interval3}
p \in [\log (\ell + 1),m(\ell + 2)]
\end{equation}
and a valuation $v'|p$ such that
$$
\| W(\ell+1,\mu) \|_{v'} \le \| \Lambda_{v'} \|_{v'},
$$
leading to
$$
1 \le \left( \prod_{v \in V_\infty} \| W(\ell+1,\mu) \|_v \right) \left\| \Lambda_{v'} \right\|_{v'}.
$$
This is the key to the lower bound for $\| \Lambda_{v'} \|_{v'}$, and the final step is to give an estimate for the product $\prod_{v \in V_\infty} \| W(\ell+1,\mu) \|_v$. Approximately it is
\begin{equation}\label{bound}
\log \left( \prod_{v \in V_\infty} \| W(\ell+1,\mu) \|_v \right) \approx (m+1)\log H + m^2 \ell \log \log \ell.
\end{equation}
The definition of $\ell$ gives a connection between $\ell$ and $H$, enabling us to write the bound \eqref{bound} and the interval \eqref{interval3} solely in terms of $H$:
$$
\ell \log \log \ell \approx \frac{\log \log \log H}{\log \log H} \cdot \log H.
$$

As the attentive reader may have noted, one crucial point in the proofs is the non-vanishing of the quantity $W(l,\mu)$. This is the part where the auxiliary parameter $\mu$ is needed. A non-vanishing determinant of the Pad\'e polynomials will ensure that for each $l \in \Z_{\ge 1}$, there exists a $\mu \in \{ 0,1,\ldots,m \}$ such that $W(l,\mu) \neq 0$.

\section{Pad\'e approximations}

Let $m \in \Z_{\ge 1}$, $\overline{l} = (l_1, \ldots, l_m)^T \in \Z_{\ge 1}^m$, and $L := \sum_{j=1}^m l_j$. For a given vector $\overline{\beta} = (\beta_1, \ldots, \beta_m )^T$, define the numbers $\sigma_i = \sigma_i \!\left( \overline{l}, \overline{\beta} \right)$ by the equation
\begin{equation}\label{sigmadef}
\prod_{j=1}^m (\beta_j -w)^{l_j} = \sum_{i=0}^L \sigma_i w^i.
\end{equation}
Then, by the binomial theorem,
$$
\sigma_i \!\left( \overline{l}, \overline{\beta} \right) = (-1)^i \sum_{i_1 + \ldots + i_m =i} \binom{l_1}{i_1} \cdots \binom{l_m}{i_m} \cdot \beta_1^{l_1-i_1} \cdots \beta_m^{l_m-i_m}.
$$
\begin{lemma}
We have
\begin{equation}\label{sigmaiproperty}
\sum_{i=0}^L \sigma_i i^k \beta_j^i =0
\end{equation}
for all $j \in \{ 1, \ldots, m \}$, $k \in \{0,1,\ldots,l_j-1\}$.
Moreover, when $\overline{\beta} = (\beta_1, \ldots, \beta_m )^T \in \mathbb{K}^m$ and $\| \cdot \|_v$ is any Archimedean absolute value of the field $\mathbb{K}$, we have
\begin{equation}\label {property2}
\sum_{i=0}^{L} \|\sigma_i\|_v t^i \le \prod_{j=1}^m (\|\beta_j\|_v+t)^{l_j}
\end{equation}
for $t \ge 0$.
\end{lemma}

\begin{proof}
It is not too hard to deduce that
\begin{equation}\label{deriv}
\left( x \frac{\mathrm{d}}{\mathrm{d}x} \right)^n f(x) = \sum_{i=1}^n a_{n,i} x^i \left( \frac{\mathrm{d}}{\mathrm{d}x} \right)^i f(x),
\end{equation}
where the coefficients $a_{n,i}$ satisfy the recursions
$$
\begin{cases}
a_{n,1}=1; \\
a_{n,i} = a_{n-1,i-1} + i a _{n-1,i}, \quad i = 2,\ldots,n-1; \\
a_{n,n}=1
\end{cases}
$$
for all $n \in \Z_{\ge 1}$. Let now $j \in \{1,\ldots,m\}$. For $k=0$, the claim \eqref{sigmaiproperty} follows directly from the definition \eqref{sigmadef}. For $k \in \{1,\ldots,l_j-1\}$, we use \eqref{sigmadef} and \eqref{deriv}:
\begin{equation*}
\begin{split}
\sum_{i=0}^L \sigma_i i^k \beta_j^i &= \left( w \frac{\mathrm{d}}{\mathrm{d}w} \right)^k \prod_{i=1}^m (\beta_i -w)^{l_i} \bigg|_{w=\beta_j} =0
\end{split}
\end{equation*}
because $\left( \frac{\mathrm{d}}{\mathrm{d}w} \right)^h \prod_{i=1}^m (\beta_i -w)^{l_i} \Big|_{w=\beta_j} =0$ for all $h \in \{1,\ldots,k\}$.

Property \eqref{property2} follows simply from the expansion of $\sigma_i$ and the triangle inequality:
\begin{align*}
\sum_{i=0}^{L} \|\sigma_i\|_v t^i &\le \sum_{i=0}^L \left( \sum_{i_1 + \ldots + i_m =i} \left\| \binom{l_1}{i_1} \cdots \binom{l_m}{i_m} \right\|_v \cdot \|\beta_1\|_v^{l_1-i_1} \cdots \|\beta_m\|_v^{l_m-i_m} \right) t^i \\
&\le \sum_{i=0}^L \left( \sum_{i_1 + \ldots + i_m =i} \binom{l_1}{i_1} \cdots \binom{l_m}{i_m} \cdot \|\beta_1\|_v^{l_1-i_1} \cdots \|\beta_m\|_v^{l_m-i_m} \right) t^i \\
&= \prod_{j=1}^m (\left\|\beta_j\right\|_v+t)^{l_j}
\end{align*}
when $t \ge 0$.
\end{proof}

\subsection{Generalised factorial series}

When $l_1=l_2=\ldots=l_m$, the following theorem is a particular case of Theorem 2.2 in \cite{Ma11}. Due to the special nature of the function \eqref{genfacser}, however, we don't need to restrict the parameters $l_j$.

\begin{theorem}\label{Pade}
Let $G(t) = \sum_{n=0}^\infty [P]_n t^n$, where $P(x)$ is a polynomial of degree one and $[P]_n = \prod_{k=0}^{n-1} P(k)$. Let $\mu \in \Z_{\ge 0}$ and set
$$
A_{\overline{l}, \mu, 0}(t) = \sum_{i=0}^L \frac{\sigma_i \!\left( \overline{l}, \overline{\beta} \right)}{[P]_{i+\mu}} t^{L-i}.
$$
Then there exist polynomials $A_{\overline{l}, \mu, j}(t)$ and remainders $R_{\overline{l}, \mu, j}(t)$, $j=1,\ldots,m$, such that
\begin{equation}\label{Padeformula}
A_{\overline{l}, \mu, 0}(t) G(\beta_jt) - A_{\overline{l}, \mu, j}(t) = R_{\overline{l}, \mu, j}(t),
\end{equation}
where
\begin{equation}\label{degreesorders}
\begin{cases}
\deg A_{\overline{l}, \mu, 0}(t) = L, \\
\deg A_{\overline{l}, \mu, j}(t) \le L+\mu-1, \\
\ord R_{\overline{l}, \mu, j}(t) \ge L+\mu+l_j.
\end{cases}
\end{equation}
\end{theorem}

\begin{proof}
Writing
$$
A_{\overline{l}, \mu, 0}(t) = \sum_{h=0}^L \frac{\sigma_{L-h} \!\left( \overline{l}, \overline{\beta} \right)}{[P]_{L-h+\mu}} t^h,
$$
we have
$$
A_{\overline{l}, \mu, 0}(t) G(\beta_jt) = \sum_{N=0}^\infty r_{N,j} t^N,
$$
where
\begin{equation}\label{remaindercoef}
r_{N,j} = \sum_{n+h=N} \sigma_{L-h} \!\left( \overline{l}, \overline{\beta} \right) \cdot  \frac{[P]_n}{[P]_{L-h+\mu}} \cdot \beta_j^n = \sum_{h=0}^{\min \{L,N\}} \sigma_{L-h} \!\left( \overline{l}, \overline{\beta} \right) \cdot \frac{[P]_{N-h}}{[P]_{L-h+\mu}} \cdot \beta_j^{N-h}.
\end{equation}
When $N=L+\mu+a$, $0 \le a \le l_j-1$, then
\begin{equation*}
\begin{split}
r_{N,j} &= \beta_j^{\mu+a} \sum_{h=0}^L \sigma_{L-h} \!\left( \overline{l}, \overline{\beta} \right) \left( \prod_{k=1}^a P(L+\mu-h-1+k) \right) \beta_j^{L-h} \\
&= \beta_j^{\mu+a} \sum_{i=0}^L \sigma_i \!\left( \overline{l}, \overline{\beta} \right) \left( \prod_{k=1}^a P(i+\mu-1+k) \right) \beta_j^i.
\end{split}
\end{equation*}
(Note that the product above equals $1$ when $a=0$.)
Since $\deg P(x) =1$, we may write
$$
\prod_{k=1}^a P(i+\mu-1+k) = \sum_{k=0}^a p_k i^k,
$$
where the coefficients $p_k$ do not depend on $i$. Hence
$$
r_{N,j} = \beta_j^{\mu+a} \sum_{i=0}^L \sigma_i \!\left( \overline{l}, \overline{\beta} \right) \left( \sum_{k=0}^a p_k i^k \right) \beta_j^i 
= \beta_j^{\mu+a} \sum_{k=0}^a p_k \sum_{i=0}^L \sigma_i \!\left( \overline{l}, \overline{\beta} \right) i^k \beta_j^i = 0
$$
due to \eqref{sigmaiproperty}. Thus we can choose
\begin{equation}\label{Almuj}
A_{\overline{l}, \mu, j}(t) = \sum_{N=0}^{L+\mu-1} r_{N,j} t^N
\end{equation}
and
\begin{equation}\label{Rlmuj}
R_{\overline{l}, \mu, j}(t) = \sum_{N=L+\mu+l_j}^\infty r_{N,j} t^N.
\end{equation}
\end{proof}

\subsection{Euler's factorial series}

To prove Theorem \ref{mainresult}, we need approximations to the series $F(\alpha_j t)$.
Thus we choose $P(x)=1+x$ and $\overline{\beta}= \overline{\alpha} = \left( \alpha_1, \ldots, \alpha_m \right)^T$, and set $l_j =l \in \Z_{\ge 1}$ for all $j \in \{1, \ldots, m\}$. Theorem \ref{Pade} gives
\begin{equation*}
A_{\overline{l},\mu,0}(t)=\sum_{i=0}^{ml} \frac{\sigma_i}{(i+\mu)!} t^{ml-i}, \quad \sigma_i = \sigma_i \left(\overline{l}, \overline{\alpha} \right),
\end{equation*}
and, directly by \eqref{Almuj} and \eqref{remaindercoef},
$$
A_{\overline{l}, \mu, j} (t) = \sum_{N=0}^{ml+\mu-1} t^N \sum_{h=0}^{\min \{ml,N\}} \sigma_{ml-h} \cdot \frac{(N-h)!}{(ml-h+\mu)!} \cdot \alpha_j^{N-h}, \quad j\in \{1,\ldots,m\}.
$$
Similarly by \eqref{Rlmuj} and \eqref{remaindercoef}, for $N=(m+1)l+\mu+k$, $k \in \N$, we have
\begin{equation*}
\begin{split}
r_{N,j} &=\sum_{h=0}^{ml} \sigma_{ml-h} \cdot \frac{((m+1)l+\mu+k-h)!}{(ml-h+\mu)!} \cdot \alpha_j^{(m+1)l+\mu+k-h}\\
&=\alpha_j^{l+\mu+k}\sum_{i=0}^{ml} \sigma_{i} \cdot \frac{(i+\mu+l+k)!}{(i+\mu)!} \cdot \alpha_j^{i}\\
&=l!k! \binom{l+k}{k} \alpha_j^{l+\mu+k}\sum_{i=0}^{ml} \sigma_{i} \binom{i+\mu+l+k}{i+\mu} \alpha_j^{i},
\end{split}
\end{equation*}
so that
$$
R_{\overline{l}, \mu, j}(t) = l! t^{(m+1)l+\mu} \sum_{k=0}^\infty t^k k! \binom{l+k}{k} \alpha_j^{l+\mu+k}\sum_{i=0}^{ml} \sigma_{i} \binom{i+\mu+l+k}{i+\mu} \alpha_j^{i}, \quad j = 1,\ldots,m.
$$

To make the polynomials belong to $\Z_{\mathbb{K}}[t]$, we multiply everything by $(ml+\mu)!$ and denote
\begin{equation*}
B_{l,\mu,0}(t) := (ml+\mu)!\,A_{\overline{l},\mu,0}(t) = \sum_{i=0}^{ml} \sigma_i \cdot \frac{(ml+\mu)!}{(i+\mu)!} \cdot t^{ml-i},
\end{equation*}
\begin{equation*}
\begin{split}
B_{l,\mu,j}(t) := (ml+\mu)!\,A_{\overline{l},\mu,j}(t) 
= (ml+\mu)! \sum_{N=0}^{ml+\mu-1} t^N \sum_{h=0}^{\min \{ml,N\}} \sigma_{ml-h} \cdot \frac{(N-h)!}{(ml-h+\mu)!} \cdot \alpha_j^{N-h},
\end{split}
\end{equation*}
\begin{equation}\label{jaannos}
\begin{split}
S_{l,\mu,j}(t) &:= (ml+\mu)!\,R_{\overline{l},\mu,j}(t) \\
&= (ml+\mu)! l! t^{(m+1)l+\mu}\sum_{k=0}^{\infty}k! \binom{l+k}{k} \alpha_j^{l+k+\mu} t^k \sum_{i=0}^{ml} \sigma_{i} \binom{i+\mu+l+k}{i+\mu} \alpha_j^{i}.
\end{split}
\end{equation}
In this notation, the Pad\'e approximation formula in \eqref{Padeformula} may be rewritten as
\begin{equation}\label{PADE3}
B_{l,\mu,0}(t) F(\alpha_j t) - B_{l,\mu,j}(t)=S_{l,\mu,j}(t), \quad j = 1, \ldots, m.
\end{equation}

\section{Linear form and product formula}\label{sec:linform}

Let $\lambda_0, \lambda_1, \ldots, \lambda_m \in \Z_{\mathbb{K}}$ be such that at least one of them is non-zero, and denote
$$
\Lambda_v := \lambda_0 + \lambda_1 F_v(\alpha_1) + \ldots + \lambda_m F_v(\alpha_m)
$$
when $v \in V_0$. Equation \eqref{PADE3} gives
$$
s_{l,\mu,i} = b_{l,\mu,0} F_v(\alpha_i) - b_{l,\mu,i}
$$
where
$$
b_{l,\mu,i} = B_{l,\mu,i}(1), \quad i = 0,1,\ldots,m; \quad s_{l,\mu,i} = S_{l,\mu,i}(1), \quad i = 1,\ldots, m.
$$
Assume that $\Lambda_v = 0$ for all $v \in V$, where the collection $V$ satisfies condition \eqref{limsupehto}. Then also
\begin{equation*}
0 = b_{l,\mu,0} \Lambda_v 
= W + \lambda_1 s_{l,\mu,1} + \ldots + \lambda_m s_{l,\mu,m},
\end{equation*}
where
\begin{equation*}
W = W (l, \mu) := \lambda_0 b_{l,\mu,0} + \lambda_1 b_{l,\mu,1} + \ldots + \lambda_m b_{l,\mu,m} \in \Z_{\mathbb{K}}.
\end{equation*}

If $W \neq 0$, then
\begin{equation}\label{tuloyht}
\begin{split}
1 &= \prod_{v} \|W\|_v  \le \left( \prod_{v \in V_\infty} \|W\|_v \right) \prod_{v \in V} \|W\|_v\\
&\le \left( \prod_{v \in V_\infty} \| \lambda_0 b_{l,\mu,0} + \lambda_1 b_{l,\mu,1} + \ldots + \lambda_m b_{l,\mu,m} \|_v \right) \prod_{v \in V} \| -\lambda_1 s_{l,\mu,1} - \ldots - \lambda_m s_{l,\mu,m} \|_v \\
&\le \left( \prod_{v \in V_\infty} \left( \sum_{i=0}^m \| \lambda_i \|_v \right) \max_{0 \le i \le m} \{ \|b_{l,\mu,i}\|_v \} \right) \prod_{v \in V} \max_{1 \le i \le m} \{ \| s_{l,\mu,i} \|_v \}.
\end{split}
\end{equation}

Next we shall see that such a non-zero $W(l, \mu)$ actually exists.

\section{Determinant}\label{Determinant}

\begin{lemma}\label{detlemma}
When the numbers $\alpha_j$, $j \in \{1,\ldots,m\}$, are pairwise different and non-zero, we have
$$
\Delta(t) :=
\begin{vmatrix}
B_{l,0,0}(t) & B_{l,0,1}(t) & \cdots & B_{l,0,m}(t) \\
B_{l,1,0}(t) & B_{l,1,1}(t) & \cdots & B_{l,1,m}(t) \\
\vdots & \vdots & \ddots & \vdots \\
B_{l,m,0}(t) & B_{l,m,1}(t) & \cdots & B_{l,m,m}(t)
\end{vmatrix}
\neq 0.
$$
\end{lemma}

\begin{proof}
By \eqref{degreesorders}, the degrees of the entries are at most
$$
\begin{pmatrix}
ml & ml-1 & \cdots & ml-1 \\
ml & ml & \cdots & ml \\
\vdots & \vdots & \ddots & \vdots \\
ml & ml+m-1 & \cdots & ml+m-1
\end{pmatrix}.
$$
Hence
$$
\deg \Delta (t) \le (m+1)ml+\frac{(m-1)m}{2}.
$$
Column operations together with \eqref{PADE3} yield the representation
\begin{equation}\label{rep2}
\Delta(t) =
\begin{vmatrix}
B_{l,0,0}(t) & -S_{l,0,1}(t) & \cdots & -S_{l,0,m}(t) \\
B_{l,1,0}(t) & -S_{l,1,1}(t) & \cdots & -S_{l,1,m}(t) \\
\vdots & \vdots & \ddots & \vdots \\
B_{l,m,0}(t) & -S_{l,m,1}(t) & \cdots & -S_{l,m,m}(t)
\end{vmatrix}.
\end{equation}
According to \eqref{degreesorders}, the orders of the entries in \eqref{rep2} are at least
$$
\begin{pmatrix}
0 & (m+1)l & \cdots & (m+1)l \\
0 & (m+1)l+1 & \cdots & (m+1)l+1 \\
\vdots & \vdots & \ddots & \vdots \\
0 & (m+1)l+m & \cdots & (m+1)l+m
\end{pmatrix}.
$$
By expanding \eqref{rep2} by the first column we see that
$$
\ord \Delta (t) \ge \sum_{i=0}^{m-1} ((m+1)l+i) = m(m+1)l+\frac{(m-1)m}{2}.
$$
Thus
$$
\Delta (t) = bt^{m(m+1)l+\frac{(m-1)m}{2}},
$$
where the coefficient $b$ is an $m \times m$ determinant formed from the lowest term coefficients of the remainders $-S_{l,\mu,j}$, $\mu = 0,1,\ldots,m-1$, $j = 1,\ldots,m$ (corresponding to $k=0$ in \eqref{jaannos}), multiplied by the lowest term coefficient of the polynomial $B_{l,m,0} (t)$ which is $\sigma_{ml} = (-1)^{ml}$:
\begin{multline*}
b= (-1)^{ml} \cdot (-1)^m (l!)^m \left( \prod_{\mu =0}^{m-1} (ml+\mu)! \right) \left( \prod_{j=1}^m \alpha_j^l \right) \cdot \\
\begin{vmatrix}
\sum_{i=0}^{ml} \sigma_{i} \binom{i+l}{i} \alpha_1^{i} & \sum_{i=0}^{ml} \sigma_{i} \binom{i+l}{i} \alpha_2^{i} & \cdots & \sum_{i=0}^{ml} \sigma_{i} \binom{i+l}{i} \alpha_m^{i} \\
\alpha_1 \sum_{i=0}^{ml} \sigma_{i} \binom{i+1+l}{i+1} \alpha_1^{i} & \alpha_2 \sum_{i=0}^{ml} \sigma_{i} \binom{i+1+l}{i+1} \alpha_2^{i} & \cdots & \alpha_m \sum_{i=0}^{ml} \sigma_{i} \binom{i+1+l}{i+1} \alpha_m^{i} \\
\vdots & \vdots & \ddots & \vdots \\
\alpha_1^{m-1} \sum_{i=0}^{ml} \sigma_{i} \binom{i+m-1+l}{i+m-1} \alpha_1^{i} & \alpha_2^{m-1} \sum_{i=0}^{ml} \sigma_{i} \binom{i+m-1+l}{i+m-1} \alpha_2^{i} & \cdots & \alpha_m^{m-1} \sum_{i=0}^{ml} \sigma_{i} \binom{i+m-1+l}{i} \alpha_m^{i} \\
\end{vmatrix}.
\end{multline*}
It remains to show that $b \neq 0$.

Since
$$
\binom{i+\mu+l}{i+\mu} = \frac{(i+\mu+l)!}{(i+\mu)! l!} = \frac{1}{l!} (i+\mu+1)\cdots(i+\mu+l) = \frac{1}{l!} \left( i^l + \sum_{k=0}^{l-1} p_k i^k \right)
$$
for any $\mu \in \{0,1,\ldots,m-1\}$, where the coefficients $p_k$ do not depend on $i$, we get
$$
\sum_{i=0}^{ml} \sigma_{i} \binom{i+\mu+l}{i+\mu} \alpha_j^{i} = \frac{1}{l!} \left( \sum_{i=0}^{ml} \sigma_i i^l \alpha_j^i + \sum_{k=0}^{l-1} p_k \sum_{i=0}^{ml} \sigma_i i^k \alpha_j^i \right) = \frac{1}{l!} \sum_{i=0}^{ml} \sigma_i i^l \alpha_j^i
$$
for all $j = 1, \ldots, m$, $\mu = 0,1,\ldots,m-1$ by the property \eqref{sigmaiproperty}. Hence
\begin{equation*}
\begin{split}
b= \;&(-1)^{m(l+1)} \left( \prod_{\mu =0}^{m-1} (ml+\mu)! \right) \left( \prod_{j=1}^m \alpha_j^l \right) \cdot \\
&\begin{vmatrix}
\sum_{i=0}^{ml} \sigma_i i^l \alpha_1^i & \sum_{i=0}^{ml} \sigma_i i^l \alpha_2^i & \cdots & \sum_{i=0}^{ml} \sigma_i i^l \alpha_m^i \\
\alpha_1 \sum_{i=0}^{ml} \sigma_i i^l \alpha_1^i & \alpha_2 \sum_{i=0}^{ml} \sigma_i i^l \alpha_2^i & \cdots & \alpha_m \sum_{i=0}^{ml} \sigma_i i^l \alpha_m^i \\
\vdots & \vdots & \ddots & \vdots \\
\alpha_1^{m-1} \sum_{i=0}^{ml} \sigma_i i^l \alpha_1^i & \alpha_2^{m-1} \sum_{i=0}^{ml} \sigma_i i^l \alpha_2^i & \cdots & \alpha_m^{m-1} \sum_{i=0}^{ml} \sigma_i i^l \alpha_m^i \\
\end{vmatrix} \\
= \;&(-1)^{m(l+1)} \left( \prod_{\mu =0}^{m-1} (ml+\mu)! \right) \left( \prod_{j=1}^m \alpha_j^l \right) \left( \prod_{j=1}^m \left( \sum_{i=0}^{ml} \sigma_i i^l \alpha_j^i \right) \right) \prod_{1 \le i < j \le m} (\alpha_j - \alpha_i)
\end{split}
\end{equation*}
by the Vandermonde determinant formula. Here, using \eqref{sigmadef} and \eqref{deriv},
$$
\sum_{i=0}^{ml} \sigma_i i^l \alpha_j^i = \left( w \frac{\mathrm{d}}{\mathrm{d}w} \right)^l \prod_{i=1}^m (\alpha_i -w)^{l} \bigg|_{w=\alpha_j} = (-1)^l l! \alpha_j^l \prod_{\substack{i=1 \\ i \neq j}}^m (\alpha_i - \alpha_j)^{l}\neq 0
$$
for all $j = 1,\ldots,m$.
\end{proof}

\begin{lemma}\label{W-lemma}
For any given $l \in \Z_{\ge 1}$ there exists a $\mu \in \{ 0,1,\ldots,m \}$ such that $W(l,\mu) \neq 0$.
\end{lemma}

\begin{proof}
From Lemma \ref{detlemma} it follows in particular that
$$
\begin{vmatrix}
b_{l,0,0} & b_{l,0,1} & \cdots & b_{l,0,m} \\
b_{l,1,0} & b_{l,1,1} & \cdots & b_{l,1,m} \\
\vdots & \vdots & \ddots & \vdots \\
b_{l,m,0} & b_{l,m,1} & \cdots & b_{l,m,m}
\end{vmatrix}
= \Delta(1) \neq 0.
$$
We assumed that $(\lambda_0, \lambda_1, \ldots, \lambda_m)^T \neq \overline{0}$, so by linear algebra it follows that the quantity $W (l, \mu) = \lambda_0 b_{l,\mu,0} + \lambda_1 b_{l,\mu,1} + \ldots + \lambda_m b_{l,\mu,m}$ must be non-zero for some $\mu \in \{ 0,1,\ldots,m \}$.
\end{proof}

\section{Estimates for the polynomials and remainders and proof of Theorem \ref{mainresult}}\label{sec:estimates}

As the last step in proving Theorem \ref{mainresult}, we give upper bounds for the Pad\'e polynomials and remainders. Now, using the triangle inequality and property \eqref{property2} with $v|\infty$,
\begin{multline*}
\| b_{l,\mu,0} \|_v = \| B_{l,\mu,0}(1) \|_v = \left\| \sum_{i=0}^{ml} \sigma_i\frac{(ml+\mu)!}{(i+\mu)!} \right\|_v \\
\le \left\|(ml)! \binom{ml+\mu}{\mu} \right\|_v \sum_{i=0}^{ml} \|\sigma_i\|_v
\le \left\|(ml)! \binom{ml+\mu}{\mu} \right\|_v \prod_{j=1}^m (\left\| \alpha_j \right\|_v+1)^l
\end{multline*}
and
\begin{equation*}
\begin{split}
\| b_{l,\mu,j} \|_v = \;&\| B_{l,\mu,j}(1) \|_v \\
= \;&\left\| (ml+\mu)! \sum_{N=0}^{ml+\mu-1} \sum_{h=0}^{\min \{ml,N\}} \frac{(N-h)!}{(ml-h+\mu)!} \sigma_{ml-h}  \alpha_j^{N-h} \right\|_v \\
\le \;&\left\| (ml+\mu)! \right\|_v \sum_{N=0}^{ml+\mu-1} \sum_{h=0}^{\min \{ml,N\}} \left\| \frac{(N-h)!}{(ml-h+\mu)!} \right\|_v \left\| \sigma_{ml-h} \right\|_v \left\| \alpha_j \right\|_v^{N-h} \\
\le \;&\left\| (ml+\mu)! \right\|_v \sum_{N=0}^{ml+\mu-1} \sum_{h=0}^{\min \{ml,N\}} \left\| \sigma_{ml-h} \right\|_v \left( \max \left\{ 1, \left\| \alpha_j \right\|_v \right\} \right)^{N-h} \\
\le \;&\left\| (ml+\mu)! \right\|_v \sum_{N=0}^{ml+\mu-1} \sum_{h=0}^{\min \{ml,N\}} \left\| \sigma_{ml-h} \right\|_v \left( \max \left\{ 1, \left\| \alpha_j \right\|_v \right\} \right)^{ml+m-1-h} \\
\le \;&\left\| (ml+\mu)! \right\|_v \left( \max \left\{ 1, \left\| \alpha_j \right\|_v \right\} \right)^{m-1} (ml+m) \cdot \\
&\sum_{h=0}^{\min \{ml,N\}} \left\| \sigma_{ml-h} \right\|_v \left( \max \left\{ 1, \left\| \alpha_j \right\|_v \right\} \right)^{ml-h} \\
\le \;&\left\| (ml+\mu)! \right\|_v \left( \max \left\{ 1, \left\| \alpha_j \right\|_v \right\} \right)^{ml} (ml+m) \prod_{i=1}^m \left(\|\alpha_i\|_v + \max \left\{ 1, \left\| \alpha_j \right\|_v \right\} \right)^l \\
\end{split}
\end{equation*}
for all $j=1,\ldots,m$, $\mu = 0,1,\ldots,m$.

We still need non-Archimedean estimates for the remainders, so let now $v \in V_0$. Then
\begin{equation*}
\begin{split}
\| s_{l,\mu,j} \|_v &= \| S_{l,\mu,j}(1) \|_v \\
&= \left\| (ml+\mu)! l! \sum_{k=0}^{\infty}k! \binom{l+k}{k} \alpha_j^{l+k+\mu} \sum_{i=0}^{ml} \sigma_{i} \binom{i+\mu+l+k}{i+\mu} \alpha_j^{i} \right\|_v \\
&\le \| (ml+\mu)! l! \|_v \left\| \alpha_j \right\|_v^l.
\end{split}
\end{equation*}
for all $j=1,\ldots,m$, $\mu = 0,1,\ldots,m$.

So, recalling property \eqref{valuationproperty} of our normalised valuations, the expression in \eqref{tuloyht} becomes
\begin{equation}\label{lastestimate}
\begin{split}
&\left( \prod_{v \in V_\infty} \left( \sum_{i=0}^m \| \lambda_i \|_v \right) \max_{0 \le i \le m} \{ \|b_{l,\mu,i}\|_v \} \right) \prod_{v \in V} \max_{1 \le i \le m} \{ \| s_{l,\mu,i} \|_v \} \\
\le \;&\Bigg( \prod_{v \in V_\infty} \left( \sum_{i=0}^m \| \lambda_i \|_v \right) (ml+m) \left\| (ml+m)! \right\|_v \left( \max_{1 \le j \le m} \left\{ 1, \left\| \alpha_j \right\|_v \right\} \right)^{ml} \cdot \\
&\prod_{i=1}^m \left(\|\alpha_i\|_v + \max_{1 \le j \le m} \left\{ 1, \left\| \alpha_j \right\|_v \right\} \right)^l \Bigg) \cdot \prod_{v \in V} \| (ml)! l! \|_v  \left( \max_{1 \le j \le m} \left\{ \left\| \alpha_j \right\|_v \right\} \right)^l \\
\le \;&\left( \prod_{v \in V_\infty} \left( \sum_{i=0}^m \| \lambda_i \|_v \right) \right) c_2^l (ml+m)^\kappa (ml+m)! \cdot \prod_{v \in V} \| (ml)! l! \|_v,
\end{split}
\end{equation}
where
\begin{equation*}
\begin{split}
c_2= &\left( \prod_{v \in V_\infty} \left( \left( \max_{1 \le j \le m} \left\{ 1, \left\| \alpha_j \right\|_v \right\} \right)^{m}  \prod_{i=1}^m \left(\|\alpha_i\|_v + \max_{1 \le j \le m} \left\{ 1, \left\| \alpha_j \right\|_v \right\} \right) \right) \right) \cdot \\
&\prod_{v \in V} \max_{1 \le j \le m} \left\{ \left\| \alpha_j \right\|_v \right\}.
\end{split}
\end{equation*}

\begin{proof}[Proof of Theorem \ref{mainresult}]
In Section \ref{Determinant} we saw that for every $l \in \Z_{\ge 1}$, there exists a $\mu \in \{0,1,\ldots,m\}$ such that $W=W (l,\mu) \neq 0$. Hence the estimate in \eqref{tuloyht} holds for infinitely many $W(l, \mu)$, so that our assumption $\Lambda_v = 0$ for all $v \in V$ and estimates \eqref{tuloyht} and \eqref{lastestimate} lead to
$$
1 \le \left( \prod_{v \in V_\infty} \left( \sum_{i=0}^m \| \lambda_i \|_v \right) \right) c_2^l (ml+m)^\kappa (ml+m)! \prod_{v \in V} \| (ml)! l! \|_v
$$
which holds for infinitely many $l$. This is a contradiction with condition \eqref{limsupehto}, and thus there must exist a valuation $v' \in V$ such that $\Lambda_{v'} \neq 0$.
\end{proof}

\section{Lower bound: proof of Theorem \ref{mainresult2}}

\subsection{Product formula again}

The fundamental product formula \eqref{prodformula} is the starting point for the proof of our second theorem as well. We repeat Section \ref{sec:linform} with a slightly more refined assumption. First we need some notation though.

Let $m \in \Z_{\ge 1}$ and $\log H \ge se^s$, where
\begin{equation}\label{chooses}
s = \max \left\{ e^\kappa +1, c_1+1, (m+3)^2+1 \right\},
\end{equation}
$$
\kappa = [\K : \Q],
$$
$$
c_1= \prod_{v \in V_\infty} \left( \left( \max_{1 \le j \le m} \left\{ 1, \left\| \alpha_j \right\|_v \right\} \right)^m \prod_{i=1}^m \left(\|\alpha_i\|_v + \max_{1 \le j \le m} \left\{ 1, \left\| \alpha_j \right\|_v \right\} \right) \right).
$$
Suppose that $\lambda_0, \lambda_1, \ldots, \lambda_m \in \Z_{\mathbb{K}}$ are such that at least one of them is non-zero and
\begin{equation*}
\prod_{v \in V_\infty} \max_{0 \le i \le m} \left\{ \| \lambda_i \|_v \right\} \le H.
\end{equation*}
Define
\begin{equation}\label{Nldef}
\begin{split}
N(l) := \;&\log H + \Bigg( 2(m+1) + \frac{2m}{l} + \frac{\log c_1}{\log \log l} + \frac{1}{\log \log l} + \frac{\left(\kappa-\frac{1}{2}\right) \log l}{l \log \log l} \\
&+ \frac{\kappa \log m}{l \log \log l} + \frac{\kappa \log (m+1)}{l \log \log l} + \frac{\kappa}{l^2 \log \log l} \Bigg) l \log \log l - l \log l
\end{split}
\end{equation}
and let
\begin{equation}\label{elldef}
\ell := \max \left\{ l \in \Z_{\ge 2} \; | \; N(l) \ge 0 \right\}.
\end{equation}
Denote, as before,
$$
\Lambda_v = \lambda_0 + \lambda_1 F_v(\alpha_1) + \ldots + \lambda_m F_v(\alpha_m).
$$
We saw in Section \ref{sec:linform} that
\begin{equation*}
b_{l,\mu,0} \Lambda_v 
= W + \lambda_1 s_{l,\mu,1} + \ldots + \lambda_m s_{l,\mu,m},
\end{equation*}
where
\begin{equation*}
W = W(l,\mu) = \lambda_0 b_{l,\mu,0} + \lambda_1 b_{l,\mu,1} + \ldots + \lambda_m b_{l,\mu,m} \in \Z_{\mathbb{K}}.
\end{equation*}

By Lemma \ref{W-lemma} we know that $W(\ell + 1,\mu) \neq 0$ for some $\mu \in \{0,1,\ldots,m\}$. Assume that
$$
\left\| b_{\ell + 1,\mu,0} \Lambda_v \right\|_v < \left\| \lambda_1 s_{\ell + 1,\mu,1} + \ldots + \lambda_m s_{\ell + 1,\mu,m} \right\|_v
$$
for all $v | p$, $p \in [\log (\ell + 1),m(\ell + 2)] \cap \mathbb{P}$. (The intersection certainly is non-empty due to  Bertrand's postulate. As for the choice of this interval, see Remark \ref{intervalremark}.) Then
\begin{align*}
\|W(\ell + 1,\mu)\|_v &= \left\| b_{\ell + 1,\mu,0} \Lambda_v - \left(\lambda_1 s_{\ell + 1,\mu,1} + \ldots + \lambda_m s_{\ell + 1,\mu,m} \right) \right\|_v \\
&= \left\| \lambda_1 s_{\ell + 1,\mu,1} + \ldots + \lambda_m s_{\ell + 1,\mu,m} \right\|_v
\end{align*}
for all $v | p$, $p \in [\log (\ell + 1),m(\ell + 2)] \cap \mathbb{P}$.
Hence, using the estimates made in Section \ref{sec:estimates} together with property \eqref{valuationproperty},
\begin{equation}\label{tuloyht2}
\begin{split}
1 = \;&\prod_{v} \|W(\ell + 1,\mu)\|_v  \\
\le \;&\left( \prod_{v \in V_\infty} \|W\|_v \right) \prod_{p \in [\log (\ell + 1),m(\ell + 2)]} \prod_{v|p} \|W\|_v\\
= \;&\left( \prod_{v \in V_\infty} \left\| \sum_{i=0}^m \lambda_i b_{\ell + 1,\mu,i} \right\|_v \right) \prod_{p \in [\log (\ell + 1),m(\ell + 2)]} \prod_{v|p} \left\| \lambda_1 s_{\ell + 1,\mu,1} + \ldots + \lambda_m s_{\ell + 1,\mu,m} \right\|_v\\
\le \;&\left( \prod_{v \in V_\infty} (m+1) \max_{0 \le i \le m} \left\{\| \lambda_i \|_v \right\} \max_{0 \le i \le m} \{ \|b_{\ell + 1,\mu,i}\|_v \} \right) \cdot \\
&\prod_{p \in [\log (\ell + 1),m(\ell + 2)]} \prod_{v|p} \max_{1 \le i \le m} \{ \| s_{\ell + 1,\mu,i} \|_v \} \\
\le \;&(m+1)^\kappa H \Bigg( \prod_{v \in V_\infty} \Bigg( (m(\ell+1)+m) \left\| (m(\ell+1)+\mu)! \right\|_v \cdot \\
&\left( \max_{1 \le j \le m} \left\{ 1, \left\| \alpha_j \right\|_v \right\} \right)^{m(\ell+1)} \prod_{i=1}^m \left(\|\alpha_i\|_v + \max_{1 \le j \le m} \left\{ 1, \left\| \alpha_j \right\|_v \right\} \right)^{\ell+1} \Bigg) \Bigg) \cdot \\
&\prod_{p \in [\log (\ell + 1),m(\ell + 2)]} \prod_{v|p} \| (m(\ell + 1)+\mu)! (\ell + 1)! \|_v \\
\le &(m+1)^{\kappa} (m(\ell+1)+m)^\kappa H  c_1^{\ell + 1}(m(\ell + 1)+\mu)! \cdot \\
&\prod_{p \in [\log (\ell + 1),m(\ell + 2)]} | (m(\ell + 1)+\mu)! (\ell + 1)! |_p \\
= \;&\frac{(m+1)^{\kappa} (m(\ell+2))^\kappa H  c_1^{\ell + 1} \prod_{p \in [\log (\ell + 1),m(\ell + 2)]} | (\ell + 1)! |_p}{\prod_{p < \log (\ell + 1)} |(m(\ell + 1)+\mu)! |_p} =: \Omega,
\end{split}
\end{equation}
where we utilised the fact that $\# V_\infty \le \kappa = [\K : \Q]$. The last equality is due to the product formula and the fact $m(\ell+1)+\mu \le m(\ell+2)$.

\subsection{Deriving contradiction}

We are working to establish a contradiction with \eqref{tuloyht2}, so let us study the expression $\log \Omega$
more closely. First of all, we have
\begin{align*}
\Omega &=\frac{(m+1)^{\kappa} (m(\ell+2))^\kappa H  c_1^{\ell + 1} \prod_{p \in [\log (\ell + 1),m(\ell + 2)]} | (\ell + 1)! |_p}{\prod_{p < \log (\ell + 1)} |(m(\ell + 1)+\mu)! |_p}  \\
&= \frac{(m+1)^{\kappa} (m(\ell+2))^\kappa H  c_1^{\ell + 1}}{(\ell+1)! \prod_{p < \log (\ell + 1)} |(m(\ell + 1)+\mu)! (\ell+1)!|_p} \\
&\le \frac{(m+1)^{\kappa} (m(\ell+2))^\kappa H  c_1^{\ell + 1} \prod_{p < \log (\ell + 1)} p^{\frac{m(\ell + 1)+\mu + (\ell+1)}{p-1}}}{(\ell+1)!} \\
\end{align*}
because of the product formula and property \eqref{kertoma-arvio}. Recall also the Stirling formula
$$
\log n! = \left( n+ \frac{1}{2} \right) \log n - n + \log \sqrt{2 \pi} + \frac{\theta (n)}{12},  \quad 0 < \theta (n) < 1.
$$
With these equations and estimate $\mu \le m$ we get
\begin{equation}\label{logestimate}
\begin{split}
\log \Omega \le \;&\log \left( \frac{(m+1)^{\kappa} (m(\ell+2))^\kappa H  c_1^{\ell + 1} \prod_{p < \log (\ell + 1)} p^{\frac{m(\ell + 1)+\mu + (\ell+1)}{p-1}}}{(\ell+1)!} \right) \\
\le \;&\kappa \log (m+1) + \kappa \log (m(\ell+2)) + \log H + (\ell + 1) \log c_1 + \\
&\sum_{p < \log (\ell + 1)} \log p^{\frac{m(\ell + 2)+(\ell + 1)}{p-1}} - \left((\ell + 1)+\frac{1}{2}\right) \log (\ell + 1) + (\ell + 1)\\
\le \;&\kappa \log (m+1) + \kappa \log m + \kappa \log (\ell+2) + \log H + (\ell + 1) \log c_1 \\
&+ (m(\ell + 2)+(\ell + 1)) \sum_{p < \log (\ell + 1)} \frac{\log p}{p-1} - (\ell + 1) \log (\ell + 1) \\
&- \frac{1}{2} \log (\ell + 1) + (\ell + 1) \\
\le \;&\log H + \kappa \log m + \kappa \log (m+1) + \frac{\kappa}{\ell + 1} + \left(\kappa -\frac{1}{2}\right) \log (\ell + 1) \\
&+ \left( \log c_1 + \left(m+1+\frac{m}{\ell+1}\right) \sum_{p < \log (\ell + 1)} \frac{\log p}{p-1} + 1 \right) (\ell + 1) \\
&- (\ell + 1) \log (\ell + 1),
\end{split}
\end{equation}
where
$$
\log (\ell +2) < \log (\ell +1) + \frac{1}{\ell +1}
$$
by the mean value theorem.

To be able to continue, we need to know how the sum
$$
\sum_{p < x} \frac{\log p}{p-1}
$$
behaves. Help is found from \cite{Rosser} (see the corollary of Theorem 6):
\begin{lemma}\label{Rosser} \cite{Rosser}
$$
\sum_{p \le x} \frac{\log p}{p} < \log x, \quad x > 1.
$$
\end{lemma}

Since $p-1 \ge \frac{p}{2}$ for all primes $p$, it follows that
\begin{equation}\label{psumma}
\sum_{p < x} \frac{\log p}{p-1} \le 2\sum_{p < x} \frac{\log p}{p} < 2 \log x.
\end{equation}

Combining estimates \eqref{tuloyht2}, \eqref{logestimate}, and \eqref{psumma}, we have
\begin{equation*}
\begin{split}
0 \le \;&\log \Omega \\
\le \;&\log H + \kappa \log m + \kappa \log (m+1) + \frac{\kappa}{\ell + 1} + \left(\kappa -\frac{1}{2}\right) \log (\ell + 1) \\
&+ \left( \log c_1 + \left(m+1+\frac{m}{\ell+1}\right) \sum_{p < \log (\ell + 1)} \frac{\log p}{p-1} + 1 \right) (\ell + 1) - (\ell + 1) \log (\ell + 1) \\
< \;&\log H + \kappa \log m + \kappa \log (m+1) + \frac{\kappa}{\ell + 1} + \left(\kappa -\frac{1}{2}\right) \log (\ell + 1) \\
&+ \left( \log c_1 + 2\left(m+1+\frac{m}{\ell+1}\right) \log \log (\ell+1) + 1 \right) (\ell + 1) - (\ell + 1) \log (\ell + 1) \\
< \;&\log H + \Bigg( 2(m+1) + \frac{2m}{\ell+1} + \frac{\log c_1}{\log \log (\ell + 1)} + \frac{1}{\log \log (\ell + 1)} \\
&+ \frac{\left(\kappa-\frac{1}{2}\right) \log (\ell + 1)}{(\ell + 1) \log \log (\ell + 1)} + \frac{\kappa \log m}{(\ell + 1) \log \log (\ell + 1)} + \frac{\kappa \log (m+1)}{(\ell + 1) \log \log (\ell + 1)} \\
&+ \frac{\kappa}{(\ell + 1)^2 \log \log (\ell + 1)} \Bigg) (\ell + 1) \log \log (\ell + 1) - (\ell + 1) \log (\ell + 1) \\
= \;&N(\ell + 1) < 0,
\end{split}
\end{equation*}
a contradiction with \eqref{elldef}. Thus there must exist a prime
\begin{equation}\label{interval2}
p \in [\log (\ell + 1),m(\ell + 2)]
\end{equation}
and a valuation $v'|p$ such that $\left\| b_{\ell + 1,\mu,0} \Lambda_{v'} \right\|_{v'} \ge \left\| \lambda_1 s_{\ell + 1,\mu,1} + \ldots + \lambda_m s_{\ell + 1,\mu,m} \right\|_{v'}$. Then, for this valuation $v'$,
$$
\|W\|_{v'} = \left\| b_{\ell + 1,\mu,0} \Lambda_{v'} - \left(\lambda_1 s_{\ell + 1,\mu,1} + \ldots + \lambda_m s_{\ell + 1,\mu,m} \right) \right\|_{v'} \le \left\| b_{\ell + 1,\mu,0} \Lambda_{v'} \right\|_{v'} \le \left\| \Lambda_{v'} \right\|_{v'},
$$
and
\begin{equation}\label{tuloyht3}
1 = \prod_v \| W \|_v 
\le \left( \prod_{v \in V_\infty} \| W \|_v \right) \| W \|_{v'} 
\le \left( \prod_{v \in V_\infty} \| W \|_v \right) \left\| \Lambda_{v'} \right\|_{v'}.
\end{equation}

\subsection{Bounds for $\ell$}

For the final stages of the proof, we need to express the number $\ell$ in terms of the height $H$. In order to do this, we introduce the inverse function of the function $y(z)=z \log z$, $z \geq 1/e$, considered in \cite{HANCLETAL}.

\begin{lemma}\label{inverse}
\cite{HANCLETAL} The inverse function $z(y)$ of the function $y(z)= z \log z$, $z \geq 1/e$, is strictly increasing. 
Define $z_0(y)=y$ and $z_n(y)=\frac{y}{\log z_{n-1} (y)}$ for $n\in\mathbb Z_{\ge 1}$. 
Suppose $y>e$, then 
$
z_1<z_3<\cdots <z<\cdots <z_2<z_0. 
$
Thus the inverse function may be given by the infinite nested logarithm fraction
\begin{equation*}
z(y) = \lim_{n\to\infty} z_{n}(y)=\frac{y}{\log \frac{y}{\log \frac{y}{\log \cdots}}},\quad y>e. 
\end{equation*}
\end{lemma}

Another little lemma from \cite{EMS2018} gives a useful upper estimate:

\begin{lemma}\label{z-funktio} \cite{EMS2018}
If $y \ge re^r$, where $r \ge e$, then
$$
z(y) \le \left(1+\frac{\log r}{r} \right) \frac{y}{\log y}.
$$
\end{lemma}

\begin{proof}
Denote $z:= z(y)$ with $y \ge r e^r$. Then
\[
z=\frac{y}{\log z}
= \frac{y}{\log y}\frac{\log y}{\log z}
=\frac{y}{\log y}\left(1+\frac{\log\log z}{\log z}\right)
\le\frac{y}{\log y}\left(1+\frac{\log r}{r}\right),
\]
because $\log z \ge r \ge e$.
\end{proof}

Now, $N(\ell+1) < 0$ implies
\begin{equation}\label{l1lowerbound}
(\ell+1) \log (\ell+1) > \log H \ge se^s,
\end{equation}
so that (applying the $z$-function) $\ell+1 > e^s$. According to \eqref{chooses}, we have
\begin{equation}\label{ell-lowerbound}
\ell > e^s -1 \ge \max \left\{ e^{e^\kappa}, e^{c_1}, e^{(m+3)^2} \right\}.
\end{equation}
Hence, using the lower bound \eqref{ell-lowerbound} and the fact that $m \ge 1$, we may estimate from the definition of $N(l)$ in \eqref{Nldef}:
\begin{equation}\label{Nl1-arvio}
\begin{split}
0 \le \;&N(\ell) \\
< \;&\log H + \Bigg( 2(m+1) + \frac{2m}{e^{(m+3)^2}} + 1 + \frac{1}{2 \log (m+3)} + \frac{\left(\kappa-\frac{1}{2}\right) (m+3)^2}{e^{(m+3)^2} \cdot \kappa} \\
&+ \frac{\kappa \log m}{e^{(m+3)^2} \cdot \kappa} + \frac{\kappa \log (m+1)}{e^{(m+3)^2} \cdot \kappa} + \frac{\kappa}{e^{2(m+3)^2} \cdot \kappa} \Bigg) \ell \log \log \ell - \ell \log \ell \\
\le \;&\log H + \left( 2(m+1) + 1 + 0.360674 + 3 \cdot 10^{-6} \right) \ell \log \log \ell - \ell \log \ell \\
< \;&\log H + \left( 2m+3.361 \right) \ell \log \log \ell - \ell \log \ell.
\end{split}
\end{equation}
Thus
\begin{equation}\label{llogl-logH}
\ell \log \ell \left( 1 - \frac{(2m+3.361) \log \log \ell}{\log \ell} \right) \le \log H,
\end{equation}
where
$$
\frac{\log \log \ell}{\log \ell} < \frac{2 \log (m+3)}{(m+3)^2}
$$
by \eqref{ell-lowerbound}, and so
\begin{equation}\label{coefficient}
1 - \frac{(2m+3.361) \log \log \ell}{\log \ell} > 1 - \frac{(2m+3.361) \cdot 2 \log (m+3)}{(m+3)^2} > 0
\end{equation}
for all $m \ge 1$.

By inequalities \eqref{llogl-logH} and \eqref{coefficient} and the lower bound in \eqref{ell-lowerbound}, we have
$$
\frac{(m+3)^2}{(m+3)^2 - (2m+3.361) \cdot 2 \log (m+3)} \cdot \log H > \ell \log \ell > (m+3)^2 e^{(m+3)^2},
$$
so we may apply Lemma \ref{z-funktio} with $r=(m+3)^2$:
\begin{equation}\label{l1upperbound}
\begin{split}
\ell &< z \left( \frac{(m+3)^2}{(m+3)^2 - (2m+3.361) \cdot 2 \log (m+3)} \cdot \log H \right) \\
&\le \left( 1 +\frac{2 \log (m+3)}{(m+3)^2} \right) \frac{\frac{(m+3)^2}{(m+3)^2 - (2m+3.361) \cdot 2 \log (m+3)} \cdot \log H}{\log \log H}.
\end{split}
\end{equation}

\subsection{Measure}

To get the measure from \eqref{tuloyht3}, we need an upper estimate for the product $\prod_{v \in V_\infty} \| W \|_v$. Back in \eqref{tuloyht2} we estimated that
$$
\prod_{v \in V_\infty} \| W \|_v \le (m+1)^{\kappa} (m(\ell+2))^\kappa H  c_1^{\ell + 1}(m(\ell+2))!
$$
(taking into account that $\mu \le m$).
From \eqref{Nl1-arvio} it follows that
$$
\ell \log \ell < \left( 2m+3.361 \right) \ell \log \log \ell + \log H
$$
and by the mean value theorem we have
$$
\log (\ell +2) < \frac{2}{\ell} + \log \ell.
$$
With these estimates we get
\begin{equation*}
\begin{split}
&\log \left( \prod_{v \in V_\infty} \| W \|_v \right) \\
\le \;&\log \left( (m+1)^\kappa (m(\ell+2))^\kappa H c_1^{\ell+1} (m(\ell+2))! \right) \\
\le \;&\kappa \log (m+1) + \kappa \log m + \kappa \log (\ell+2) + \log H + (\ell+1) \log c_1 \\
&+ (m(\ell+2)) \log (m(\ell+2)) \\
= \;&\kappa \log (m+1) + \kappa \log m + \kappa \log (\ell+2) + \log H + \ell \log c_1 + \log c_1 \\
&+ (m \log m)\ell + m\ell \log (\ell+2) + 2m \log m + 2m \log (\ell+2) \\
\le \;&\kappa \log (m+1) + \kappa \log m + \frac{2 \kappa}{\ell} + \kappa \log \ell + \log H + \ell \log c_1 + \log c_1 + (m \log m)\ell \\
&+ m\ell \log \ell + 2m + 2m \log m + 2m \log \ell + \frac{4m}{\ell} \\
< \;&\kappa \log (m+1) + \kappa \log m + \frac{2 \kappa}{\ell} + \kappa \log \ell + \log H + \ell \log c_1 + \log c_1 + (m \log m)\ell \\
&+ m\left( \left( 2m+3.361 \right) \ell \log \log \ell + \log H \right) + 2m + 2m \log m + 2m \log \ell + \frac{4m}{\ell} \\
= \;&(m+1) \log H + \Bigg( \frac{\kappa \log (m+1)}{\ell \log \log \ell} + \frac{\kappa \log m}{\ell \log \log \ell} + \frac{2 \kappa}{\ell^2 \log \log \ell} + \frac{\kappa \log \ell}{\ell \log \log \ell} \\
&+ \frac{\log c_1}{\log \log \ell} + \frac{\log c_1}{\ell \log \log \ell} + \frac{m \log m}{\log \log \ell} + 2m^2 +3.361m + \frac{2m}{\ell \log \log \ell} \\
&+ \frac{2m \log m}{\ell \log \log \ell} + \frac{2m \log \ell}{\ell \log \log \ell} + \frac{4m}{\ell^2 \log \log \ell} \Bigg) \ell \log \log \ell \\
\end{split}
\end{equation*}
In the coefficient of $\ell \log \log \ell$, we have (using the bound \eqref{ell-lowerbound} and the fact that $m \ge 1$)
$$
\frac{m \log m}{\log \log \ell} < \frac{m \log m}{2 \log (m+3)} < \frac{m}{2}, \quad \frac{\log c_1}{\log \log \ell} < 1,
$$
and the rest of the fractions together are less than $0.0000034$. Hence
\begin{equation}\label{logW}
\log \left( \prod_{v \in V_\infty} \| W \|_v \right) \le (m+1) \log H + \left( 2m^2 +3.861m + 1.0000034 \right) \ell \log \log \ell.
\end{equation}

By \eqref{l1upperbound} and the assumption $\log H \ge se^s > (m+3)^2 e^{(m+3)^2}$, we have
\begin{align*}
\ell &< \frac{\left( 1 +\frac{2 \log (m+3)}{(m+3)^2} \right) (m+3)^2}{(m+3)^2 - (2m+3.361) \cdot 2 \log (m+3)} \cdot \frac{\log H}{2 \log (m+3) + (m+3)^2} \\
&= \frac{1}{(m+3)^2 - (2m+3.361) \cdot 2 \log (m+3)} \cdot \log H < \log H.
\end{align*}
Thus
\begin{equation}\label{loglogl1}
\log \log \ell < \log \log \log H.
\end{equation}

Let us next estimate $\left( 2m^2 +3.861m + 1.0000034 \right) \ell$, again using \eqref{l1upperbound}:
\begin{equation}\label{mainerrorterm}
\begin{split}
&\left( 2m^2 +3.861m + 1.0000034 \right) \ell \\
\le \;&\frac{\left( 2m^2 +3.861m + 1.0000034 \right) \left( 1 +\frac{2 \log (m+3)}{(m+3)^2} \right) (m+3)^2}{(m+3)^2 - (2m+3.361) \cdot 2 \log (m+3)} \cdot  \frac{\log H}{\log \log H} \\
= \;&\frac{m^2 \left( 2 +\frac{3.861}{m} + \frac{1.0000034}{m^2} \right) \left( 1 +\frac{2 \log (m+3)}{(m+3)^2} \right)}{1 - \frac{2 \cdot 2m \log (m+3)}{(m+3)^2} - \frac{2 \cdot 3.361 \log (m+3)}{(m+3)^2}} \cdot  \frac{\log H}{\log \log H} \\
< \;&114 m^2 \cdot  \frac{\log H}{\log \log H}
\end{split}
\end{equation}
since $m \ge 1$.


Combining estimates \eqref{logW}, \eqref{loglogl1}, and \eqref{mainerrorterm}, yields
$$
\log \left( \prod_{v \in V_\infty} \| W \|_v \right) < \left( (m+1) + 114 m^2 \cdot  \frac{\log \log \log H}{\log \log H} \right) \log H,
$$
so that inequality \eqref{tuloyht3} implies
$$
\left\| \Lambda_{v'} \right\|_{v'} \ge \frac{1}{\prod_{v \in V_\infty} \| W \|_v} > H^{-(m+1) - 114 m^2 \cdot  \frac{\log \log \log H}{\log \log H}}.
$$

\subsection{Infinitely many intervals}\label{InfInt}

We still need an upper estimate for $m(\ell+1)$ in terms of the height $H$ in order to write the interval \eqref{interval2} with respect to $H$. Once more we use \eqref{l1upperbound} and the assumption $\log H \ge se^s > (m+3)^2 e^{(m+3)^2}$:
\begin{equation}\label{coeff}
\begin{split}
&m(\ell +2) \\
\le \;&m \cdot \frac{\left( 1 +\frac{2 \log (m+3)}{(m+3)^2} \right) (m+3)^2}{(m+3)^2 - (2m+3.361) \cdot 2 \log (m+3)} \cdot \frac{\log H}{\log \log H} + 2m \\
= \;&m \left( \frac{1 +\frac{2 \log (m+3)}{(m+3)^2}}{1 - \frac{2 \cdot 2m \log (m+3)}{(m+3)^2} - \frac{2 \cdot 3.361 \log (m+3)}{(m+3)^2}} + \frac{2 \log \log H}{\log H} \right) \cdot \frac{\log H}{\log \log H} \\
\le \;&m \left( \frac{1 +\frac{2 \log (m+3)}{(m+3)^2}}{1 - \frac{2 \cdot 2m \log (m+3)}{(m+3)^2} - \frac{2 \cdot 3.361 \log (m+3)}{(m+3)^2}} + \frac{4 \log (m+3) + 2(m+3)^2}{(m+3)^2 e^{(m+3)^2}} \right) \cdot \frac{\log H}{\log \log H} \\
< \;&17 m \cdot \frac{\log H}{\log \log H}
\end{split}
\end{equation}
since $m \ge 1$.


By \eqref{l1lowerbound} and Lemma \ref{z-funktio} we have
$$
\log (\ell +1) > \log (z(\log H)) > \log (z_1 (\log H)) = \log \left( \frac{\log H}{\log \log H} \right).
$$
Combining this with \eqref{coeff} above leads to
\begin{equation*}
[ \log (\ell+1), m(\ell+2) ] \subseteq \left] \log \left( \frac{\log H}{\log \log H} \right), \frac{17m \log H}{\log \log H} \right[ =: I(m,H).
\end{equation*}
Letting $H$ have values in a very rapidly increasing sequence, something like $H_{i+1}=e^{e^{H_i}}$, the intervals $I(m,H_i)$ will be distinct.

This ends the proof of Theorem \ref{mainresult2}.
\begin{flushright}
\qed
\end{flushright}

\begin{remark}
The constants $114$ and $17$ can be improved by adjusting the lower bound of $\log H$, i.e. the choice of $s$ in \eqref{chooses}. For instance, taking $(m+3)^3$ instead of $(m+3)^2$ will reduce them considerably.
\end{remark}

\begin{remark}\label{intervalremark}
There is a connection between the width of the interval $I(m,H)$ and the error term in the lower bound \eqref{lowerbound}. Our choice of $\log (\ell+1)$ in the interval \eqref{interval2} results in the term $\log \log \log H$ in \eqref{lowerbound} (see \eqref{loglogl1}), improving the corresponding lower bound of Bertrand et. al. in \cite{Bertrand} for this function.
This is done at a cost, though, since our interval $I(m,H)$ is wider than theirs. Had we chosen $e^{\sqrt{\log (\ell+1)}}$ instead of $\log (\ell+1)$, we would have ended up with $\sqrt{\log \log H}$ instead of $\log \log \log H$. Then the dependence on $H$ in the error term of \eqref{lowerbound} would have been $\frac{1}{\sqrt{\log \log H}}$, just as it is in \cite{Bertrand}, and the interval $I(m,H)$ would have had $\exp \left( \sqrt{\log \left( \frac{\log H}{\log \log H} \right)} \right)$ as its lower bound, very much like in \cite{Bertrand} and \cite{Keijo}.

The best lower bound (in terms of $H$) would have been achieved by considering an interval of the form $[2, ml]$ with no dependence on $l$ in the lower bound, because the empty sum $\sum_{p < 2} \frac{\log p}{p-1}$ would not then cause an extra term in our estimates. The disappearing of $\log \log l$ from the estimates would mean that we would have $\frac{1}{\log \log H}$ instead of $\frac{\log \log \log H}{\log \log H}$ in the error term. This is in line with the exponential function (see \cite{EMS2018}). However, this result won't give us infinitely many distinct primes when $H$ grows, like Theorem \ref{mainresult2} does.
\end{remark}

\section{Corollaries and examples}\label{sec:corollaries}

\subsection{The field of rationals}

When $\K=\Q$, Theorem \ref{mainresult} reduces to:
\begin{corollary}
Let $m \in \Z_{\ge 1}$ and $\lambda_0, \lambda_1, \ldots, \lambda_m \in \Z$ where $\lambda_j \neq 0$ for at least one $j$.
Choose $m$ pairwise distinct, non-zero integers $\alpha_j \in \Z \setminus \{0\}$, $j=1,\ldots,m$. 
Suppose $P$ is a subset of the prime numbers such that
$$
\limsup_{l \to \infty}
c_2^l (ml+m) (ml+m)! \prod_{p \in P} | (ml)! l! |_p
=0,
$$
where
$$
c_2 = \left( \max_{1 \le j \le m} \left\{ 1, \left| \alpha_j \right| \right\} \right)^{m}  \left( \prod_{i=1}^m \left(|\alpha_i| + \max_{1 \le j \le m} \left\{ 1, \left| \alpha_j \right| \right\} \right) \right) \prod_{p \in P} \max_{1 \le j \le m} \left\{ \left| \alpha_j \right|_p \right\}.
$$
Then there exists a prime $p' \in P$ for which
$$
\lambda_0 + \lambda_1 F_{p'}(\alpha_1) + \ldots + \lambda_m F_{p'}(\alpha_m) \neq 0.
$$
\end{corollary}

\begin{example}
For instance, take $\alpha_1=1$ and $\alpha_2=-1$. Then, if $P \subseteq \mathbb{P}$ is such that
$$
\limsup_{l \to \infty}
4^l (2l+2)(2l+2)! \prod_{p \in P} | (2l)! l! |_p
=0,
$$
there exists a prime $p' \in P$ for which
$$
\lambda_0 + \lambda_1 F_{p'}(1) + \lambda_2 F_{p'}(-1) \neq 0.
$$
In particular, taking $\lambda_0=2a \in \Z$ and $\lambda_1=\lambda_2=-b \in \Z$, it follows that there exists a prime $p \in P$ such that
$$
a - b \sum_{n=0}^\infty (2n)! \neq 0,
$$
i.e. $\sum_{n=0}^\infty (2n)! \neq \frac{a}{b}$ for some $p' \in P$.
\end{example}

\subsection{Linear recurrences}

A sequence $(x_n)_{n=0}^\infty$ satisfies a \emph{$k$th order homogeneous linear recurrence with constant coefficients}, if, for all $n \in \Z_{\ge k}$,
$$
x_n = c_1 x_{n-1} + c_2 x_{n-2} + \ldots + c_k x_{n-k}
$$
for some $c_1,\ldots,c_k \in \C$ with $c_k \neq 0$.
If the characteristic polynomial $x^k-c_1x^{k-1}-\ldots -c_k \in \C[x]$ of this recurrence has $k$ distinct zeros $\alpha_1,\ldots,\alpha_k \in \C$, then the solution $(x_n)_{n=0}^\infty$ is given by the linear combination
$$
x_n = a_1 \alpha_1^n + \ldots + a_k \alpha_k^n, \quad n \in \Z_{\ge 0},
$$
where the coefficients $a_1, \ldots, a_k \in \C$ are determined by given initial conditions. (More about recurrences in \cite{Cull}.)

Suppose now that $c_1,\ldots,c_k \in \Z$. Then the roots $\alpha_1,\ldots,\alpha_k$ lie in a number field $\K$ of degree at most $k$, and so do the coefficients $a_1,\ldots,a_k$. Furthermore, if $\alpha_1,\ldots,\alpha_k \in \Z_\K$, then $F(\alpha_i)$, $i=1,\ldots,k$, converges for any non-Archimedean valuation $v$ of $\K$, and we have
$$
\sum_{i=1}^k a_i F_v(\alpha_i) = \sum_{i=1}^k a_i \sum_{n=0}^\infty n!\alpha_i^n = \sum_{n=0}^\infty n! \sum_{i=1}^k a_i \alpha_i^n = \sum_{n=0}^\infty n!x_n.
$$
Multiplying both sides by $d:=\lcm_{1\le i \le k} \{ \den a_i \}$ \footnote{The denominator $\den \alpha$ of an algebraic number $\alpha \in \mathbb{K}$ is the smallest positive rational integer $n$ such that $n \alpha$ is an algebraic integer.}we get a linear form with coefficients $b_i := d a_i \in \Z_\K$:
$$
\sum_{i=1}^k b_i F_v(\alpha_i) = d \sum_{n=0}^\infty n!x_n.
$$
If at least one of the coefficients $a_i$ is non-zero, it follows from Theorem \ref{mainresult} that for any $a, b \in \Z_\K$ there exists a non-Archimedean valuation $v'$ of $\K$ such that
$$
\sum_{n=0}^\infty n!x_n \neq \frac{a}{b}.
$$

\begin{example}[The Fibonacci numbers]
The Fibonacci numbers are given by the sequence
$$
f_n = \frac{1}{\sqrt{5}} (\alpha^n - \beta^n), \quad \alpha = \frac{1+\sqrt{5}}{2}, \quad \beta = \frac{1-\sqrt{5}}{2}, \quad n \in \Z_{\ge 0}.
$$
Let us work in $\Q (\sqrt{5})$ and study the series $\sum_{n=0}^\infty n! f_n$. The minimal polynomial of $\alpha$ and $\beta$ is $x^2-x-1$, so $\alpha$ and $\beta$ are algebraic integers and thus $\| \alpha \|_v, \| \beta \|_v \le 1$ for any non-Archimedean valuation $v$ of the field $\Q (\sqrt{5})$. Actually, as $\alpha \beta = -1$, we get $\|\alpha\|_v = \|\beta\|_v = 1$ for all $v \in V_0$. Hence both series $\frac{1}{\sqrt{5}}\sum_{n=0}^\infty n! \alpha^n$ and $-\frac{1}{\sqrt{5}}\sum_{n=0}^\infty n! \beta^n$ converge $v$-adically and their sum is
$$
\frac{1}{\sqrt{5}} \left(F_v(\alpha)-F_v(\beta)\right)= \frac{1}{\sqrt{5}} \left(\sum_{n=0}^\infty n! \alpha^n -\sum_{n=0}^\infty n! \beta^n \right)= \frac{1}{\sqrt{5}}\sum_{n=0}^\infty n! (\alpha^n-\beta^n) = \sum_{n=0}^\infty n! f_n.
$$

Because $x^2-5=\left(x-\sqrt{5}\right)\left(x+\sqrt{5}\right)$ in $\R[x]$, the Archimedean absolute value of $\Q$ has two extensions to $\Q (\sqrt{5})$. These are given by
$$
\left|a+b\sqrt{5}\right|_1 = \left|a+b\sqrt{5}\right|, \quad \left|a+b\sqrt{5}\right|_2 = \left|a-b\sqrt{5}\right|,
$$
where now $| \cdot |$ is the unique Archimedean extension of the Archimedean absolute value of $\Q$ to $\C$, the algebraic closure of the Archimedean completion of $\Q$.
Further,
$$
\left\|a+b\sqrt{5}\right\|_1 = \left|a+b\sqrt{5}\right|_1^{\frac{1}{2}} = \sqrt{\left|a+b\sqrt{5}\right|},
$$
$$
\left\|a+b\sqrt{5}\right\|_2 = \left|a+b\sqrt{5}\right|_2^{\frac{1}{2}} = \sqrt{\left|a-b\sqrt{5}\right|}.
$$

Let $a, b \in \Z$, $b \neq 0$ and choose $\alpha_1=\alpha$, $\alpha_2=\beta$. Then
\begin{align*}
c_2 \left( (\alpha, \beta), V \right) = \;&\left( \max \left\{ 1, \left\| \alpha \right\|_1, \left\| \beta \right\|_1 \right\} \right)^{2} 
\left( \max \left\{ 1, \left\| \alpha \right\|_2, \left\| \beta \right\|_2 \right\} \right)^{2} \cdot \\
&\left(\|\alpha\|_1 + \max \left\{ 1, \left\| \alpha \right\|_1, \left\| \beta \right\|_1 \right\} \right) 
\left(\|\beta\|_1 + \max \left\{ 1, \left\| \alpha \right\|_1, \left\| \beta \right\|_1 \right\} \right) \cdot \\
&\left(\|\alpha\|_2 + \max \left\{ 1, \left\| \alpha \right\|_2, \left\| \beta \right\|_2 \right\} \right) 
\left(\|\beta\|_2 + \max \left\{ 1, \left\| \alpha \right\|_2, \left\| \beta \right\|_2 \right\} \right) \cdot \\
&\prod_{v \in V} \max \left\{ \left\| \alpha \right\|_v, \left\| \beta \right\|_v \right\} \\
= \;& 4 \left( \frac{1+\sqrt{5}}{2} \right)^3 
\left( \sqrt{\frac{-1+\sqrt{5}}{2}} + \sqrt{\frac{1+\sqrt{5}}{2}} \right)^2 \approx 72.
\end{align*}
By taking $\lambda_0=5a \in \Z$, $\lambda_1=\lambda_2=-b\sqrt{5} \in \Z_\K \setminus \{0\}$, Theorem \ref{mainresult} gives:
\begin{corollary}
If $V$ is any collection of non-Archimedean valuations of $\Q(\sqrt{5})$ such that
$$
\limsup_{l \to \infty}
c_2^l (2l+2)(2l+2)! \prod_{v \in V} \| (2l)! l! \|_v
=0,
$$
then there exists a valuation $v' \in V$ for which
$$
a-b \sum_{n=0}^\infty n! f_n \neq 0.
$$
\end{corollary}
\end{example}

\subsection{Arithmetic progressions}

In \cite{ATL2018} the authors prove:

\begin{proposition} \cite[Theorem 3]{ATL2018}
Let $a \in \Z$, $b, \xi \in \Z \setminus \{0\}$, and $n \in \Z_{\ge 3}$ be given. Assume that $R$ is any union of the primes in $r$ residue classes in the reduced residue system modulo $n$, where $r > \frac{\varphi (n)}{2}$. Then there are infinitely many primes $p \in R$ such that $a-bF_p(\xi) \ne 0$.
\end{proposition}

Because each non-Archimedean valuation of the number field $\K$ is attached to the prime it extends, the division of primes into $\varphi (n)$ residue classes induces a division of the non-Archimedean valuations into $\varphi (n)$ classes.
How many of these classes are needed to fulfil condition \eqref{limsupehto}?


\begin{theorem}\label{ResidueThm}
Let $m \in \Z_{\ge 1}$ and $\lambda_0, \lambda_1, \ldots, \lambda_m \in \Z_{\mathbb{K}}$ where $\lambda_j \neq 0$ for at least one $j$.
Choose $m$ pairwise distinct, non-zero algebraic integers $\alpha_1, \ldots, \alpha_m \in \Z_\K$.
Let $n \in \Z_{\ge 3}$ be given. Assume that $R$ is a union of the primes in $r$ residue classes in the reduced residue system modulo $n$, where $r > \frac{m \varphi (n)}{m+1}$, and let $V=\{ v \in V_0 \; | \; v|p \ \text{for some} \ p \in R \}$. Then there exists a valuation $v' \in V$ such that
$$
\lambda_0 + \lambda_1 F_{v'}(\alpha_1) + \ldots + \lambda_m F_{v'}(\alpha_m) \neq 0.
$$
\end{theorem}

\begin{proof}
Let us show that the collection $V$ satisfies condition \eqref{limsupehto}. We shall follow the method in \cite{ATL2018}. By \cite[Lemma 1]{ATL2018} we have
$$
\log \left( \prod_{p \equiv a \pmod{n}} |l!|_p \right) = -\frac{l \log l}{\varphi (n)} + O(l \log \log l)
$$
when $n \in \Z_{\ge 3}$ and $\gcd (a,n) =1$. Using this, the fact
$$
\log ((ml+m)!) = ml \log l + O(l),
$$
and property \eqref{valuationproperty}, we get
\begin{align*}
&\quad \log \left( c_2^l (ml+m)^\kappa (ml+m)! \prod_{v \in V} \| (ml)! l! \|_v \right) \\
&= l \log c_2 + \kappa \log (m(l+1)) + \log ((ml+m)!) + \sum_{v \in V} \log \|(ml)!l!\|_v \\
&= ml \log l + O(l) + \sum_{p \in R} \sum_{v|p} \log \| (ml)! l! \|_v \\
&= ml \log l + O(l) + \sum_{p \in R} \log | (ml)! l! |_p \\
&= ml \log l + O(l) -\frac{rml \log l}{\varphi (n)} -\frac{rl \log l}{\varphi (n)} + O(l \log \log l) \\
&= \left( m -\frac{r(m+1)}{\varphi (n)} \right) l \log l + O(l \log \log l) \\
&\overset{l \to \infty}{\to} -\infty,
\end{align*}
because the coefficient $\left( m -\frac{r(m+1)}{\varphi (n)} \right)$ is negative. The result follows from Theorem \ref{mainresult}.
\end{proof}

\section*{Acknowledgements}

The work of the author was supported by the University of Oulu Scholarship Foundation and the Vilho, Yrj\"o and Kalle V\"ais\"al\"a Foundation.



\begin{thebibliography}{00}

\bibitem{Bachman} G. Bachman: Introduction to $p$-adic numbers and valuation theory, Academic Press, New York, 1964.

\bibitem{Bertrand} D. Bertrand, V. G. Chirski\u\i, J. Yebbou: Effective estimates for global relations on Euler-type series, \emph{Ann. Fac. Sci. Toulouse Math.} (6) 13 (2004), no.~2, 241--260.

\bibitem{Chirskii1989} V. G. Chirski\u\i: Non-trivial global relations, \emph{Vestn. Mosk. Univ. Ser. I Mat. Mekh.} 44 (1989), no. 5, 33--36; English translation in \emph{Moscow Univ. Math. Bull.} 44 (1989), no. 5, 41--44.

\bibitem{Chirskii1990} V. G. Chirski\u\i: Global relations, \emph{Mat. Zametki} 48, No. 2, (1990), 123--127; English translation in \emph{Math. Notes} 48 (1990), no. 1-2, 795--798.

\bibitem{Chirskii1992} V. G. Chirski\u\i: On Algebraic Relations in Non-Archimedean Fields, \emph{Funct. Anal. Appl.} 26 (1992), no.~2, 108--115.

\bibitem{Cull} P. Cull, M. Flahive, R. Robson: Difference Equations, Springer, New York, 2005.

\bibitem{EMS2018} A-M. Ernvall-Hyt\"onen, T. Matala-aho, L. Sepp\"al\"a: On Mahler's Transcendence Measure for $e$, \emph{Constr. Approx.} (2018). \url{https://doi.org/10.1007/s00365-018-9429-3}.

\bibitem{ATL2018} A-M. Ernvall-Hyt\"onen, T. Matala-aho, L. Sepp\"al\"a: Euler's divergent series in arithmetic progressions, 2018. Submitted; preprint available as \href{https://arxiv.org/abs/1809.03859}{arXiv:1809.03859 [math.NT]}.

\bibitem{HANCLETAL} J. Han\v{c}l, M. Leinonen, K. Lepp\"al\"a, T. Matala-aho: Explicit irrationality measures for continued fractions,  \textit{J. Number Theory} 132 (2012), 1758--1769.

\bibitem{Lang} S. Lang: Algebraic Number Theory, Springer-Verlag, New York, 1986.

\bibitem{MatZud} T. Matala-aho, W. Zudilin: Euler's factorial series and global relations, \emph{J. Number Theory} 186 (2018), 202--210.

\bibitem{Ma11} T. Matala-aho: Type II Hermite--Pad\'e approximations of generalized hypergeometric series, \emph{Constr. Approx.} 33 (2011), 289--312.

\bibitem{Rosser} J. B. Rosser, L. Schoenfeld: Approximate formulas for some functions of prime numbers, \emph{Illinois J. Math.} 6 (1962), 6--94.

\bibitem{Keijo} K. V\"a\"an\"anen: On Pad\'e approximations and global relations of some Euler-type series, \emph{Int. J. Number Theory} 14, no. 8, (2018), 2303--2315.


\end{thebibliography}
\end{document}